\documentclass[12pt]{article}
\usepackage[reqno]{amsmath}
\usepackage{amsthm}
\usepackage{amssymb,amscd,latexsym}
\usepackage{paralist}
\usepackage{bm}
\usepackage[bbgreekl]{mathbbol}
\usepackage[all]{xy}

\newcommand{\F}{{\mathbb{F}}}

\newcommand{\N}{{\mathbb{N}}}

\newcommand{\Q}{{\mathbb{Q}}}

\newcommand{\Z}{{\mathbb{Z}}}

\newcommand{\id}{\mathrm{id}}

\newcommand{\pr}{\mathrm{pr}}

\renewcommand{\mod}{\;\mathrm{mod}\;}
\newcommand{\ord}{\mathrm{ord}}

\newcommand{\Ker}{\mathrm{Ker}\,}

\newcommand{\talpha}{\tilde{\alpha}}

\newcommand{\Gh}{{\mathcal G}}

\newcommand{\ovarphi}{\overline{\varphi}}

\newcommand{\silo}{\xrightarrow{\sim}}

\newcommand{\verk}{\mbox{\scriptsize $\,\circ\,$}}

\newtheorem{theorem}{Theorem}[section]
\newtheorem{lemma}[theorem]{Lemma}

\newtheorem*{rem}{Remark}

\begin{document}
\title{$\Z R$ and rings of Witt vectors $W_S (R)$}
\author{Christopher Deninger\footnote{supported by CRC 878} \and Anton Mellit\footnote{supported by the Austrian Science Fund (FWF) through the START-Project Y963-N35 of Michael Eichmair}}
\date{\ }
\maketitle
\section{Introduction} \label{sec:1}
For a commutative ring $R$, let $\Z R$ be the monoid algebra of $(R, \cdot)$. Let $T$ be a divisor stable subset of the natural numbers $\N$ and consider the ring $W_T (R)$ of $T$-Witt vectors. The Teichm\"uller map $R \to W_T (R)$ is multiplicative and hence extends uniquely to a ring homomorphism $\alpha_T : \Z R \to W_T (R)$. We are interested in the kernel of this map. If $R$ has no $T$-torsion, the ghost map $\Gh_T : W_T (R) \to R^T$ is injective and hence $\ker \alpha_T = \ker (\Gh_T \verk \alpha_T)$ consists of the elements $x = \sum_{r \in R} n_r [r] \in \Z R$ which satisfy the equations
\[
\sum_{r \in R} n_r r^{\nu} = 0 \quad \text{for} \; \nu \in T \; .
\]
If $R$ is a perfect $\F_p$-algebra and $T = \{ 1 , \ldots , p^{n-1} \}$ then $\ker \alpha_T = I^n$, where $I$ is the kernel of the map $\Z R \to R$ sending $x$ to $\sum n_r r$. In this case the induced map $\Z R / I^n \silo W_T (R)$ is actually an isomorphism, see \cite{CD1}. More generally, $\ker \alpha_T$ is known for all $\F_p$-algebras $R$ with injective Frobenius map, see \cite{CD2} Theorem 7.1. 

In the present note we use $\lambda$-ring structures to describe $\ker \alpha_T$ for more general rings $R$ and certain subsets $T = S_N$ obtained as follows. Fix a divisor stable subset $S$ which is also multiplicatively closed. Thus $S$ consists of all natural numbers whose prime divisors lie in a given set of prime numbers. Fix some $1 \le N \le \infty$ and set $S_N := \{ \nu \in S \mid \nu < N \}$. For a $\Z_S = \Z [p^{-1} , p \notin S]$-algebra $R$ all its (truncated) Witt rings are $\Z_S$-algebras as well, see \cite{H}, Lemma 1.9. Thus the Teichm\"uller map $R \to W_{S_N} (R)$ induces a homomorphism of $\Z_S$-algebras:
\begin{equation}
\label{eq:1}
\alpha_{S_N} : \Z_S R = \Z R \otimes_{\Z} \Z_S \longrightarrow W_{S_N} (R) \; .
\end{equation}
Let $\pi : \Z_S R \to R$ be the map sending $\sum n_r [r]$ to $\sum n_r r$, and for an integer $n \ge 1$ write $n_S = \prod_{p \in S} p^{\ord_p (n)}$. Then $n_S \in S$ since $S$ is multiplicatively closed. The following result holds:

\begin{theorem}
\label{t1}
Consider the unique (special) $\lambda$-ring structure $(\lambda^n_S)$ on $\Z_S R$ whose associated Adams operators $\psi^n_S : \Z_S R \to \Z_S R$ are determined by the formula $\psi^n_S [r] = [r]^{n_S}$ for $r \in R$. Then we have
\[
\ker \alpha_{S_N} = \{ x \in \Z_S R \mid \pi \lambda^n_S (x) = 0 \; \text{for} \; 1 \le n < N \} \; .
\]
\end{theorem}

Existence and uniqueness of the special $\lambda$-ring structure are special cases of a classical result \cite{W}, Proposition 1.2.

\begin{lemma}[Wilkerson]
\label{t2}
Let $B$ be commutative ring without $\Z$-torsion and for $n \ge 1$ let $\psi^n$ be a family of ring endomorphisms of $B$ such that $\psi^1 = \id$ and $\psi^n \verk \psi^m = \psi^{nm}$ and such that $\psi^p (b) \equiv b^p \mod p B$ for all $b \in B$ and all prime numbers $p$. Then there is a unique structure of a (special) $\lambda$-ring on $B$ whose Adams operators are the given maps $\psi^n$.
\end{lemma}

The main ingredient in the proof of Theorem \ref{t1} is a (unital) ring homomorphism $\ovarphi_S : W_S (R) \to W (R)$ for $\Z_S$-algebras $R$ which splits the canonical projection $W (R) \to W_S (R)$.

Adapting a method of Dwork in the theory of $p$-adic formal power series, we obtain explicit albeit complicated formulas for the operations $\lambda^n_S$ in Theorem \ref{t1}. They are given as follows. Let $\mu$ be the Moebius function and for $x \in \Z_S R$ and $k \in S$ set:
\begin{equation}
\label{eq:2}
\tau_k (x) = k^{-1} \sum_{d \mid k} \mu (d) \psi^{k/d}_S (x) \; .
\end{equation}
For an $S$-tuple $\nu = (\nu_k)_{k \in S}$ with all $\nu_k \ge 0$ we write
\[
{\tau (x) \choose \nu} = \prod_{k \in S} {\tau_k (x) \choose \nu_k} \in \Z R \otimes_{\Z} \Q \; .
\]
Set $|\nu| = \sum_{k \in S} \nu_k$ and $\| \nu \| = \sum_{k \in S} k \nu_k$. 

\begin{theorem}
\label{t3}
In the situation of Theorem \ref{t1}, for $n \ge 1$ and $x \in \Z_S R$ the following explicit formula holds in $\Z_S R$:
\[
(-1)^n \lambda^n_S (x) = \sum_{\| \nu \| = n} (-1)^{|\nu|} {\tau (x) \choose \nu} \; .
\]
\end{theorem}

The methods work in the more general situation where instead of $\Z_S R$ we start with a $\Z$-torsionfree $\Z_S$-algebra $B$ which is equipped with commuting Frobenius lifts $\psi^p_S$ for all primes $p \in S$. Setting $\psi^p_S = \id$ for $p \notin S$, we show that the corresponding $\lambda$-ring structure is given by the same formula as in Theorem \ref{t3}. Moreover any homomorphism $\pi : B \to R$ into a $\Z_S$-algebra $R$ factors canonically over maps $\alpha_{S_N} : B \to W_{S_N} (R)$ and the kernel of $\alpha_{S_N}$ is obtained as in Theorem \ref{t1}.

For background on the theory of Witt vectors we refer to \cite{H}, \cite{R}, \cite{BW} and \cite{CD2}. The latter approaches avoid universal polynomials.

Both authors would like to thank the HIM in Bonn, where this work originated for its support.
\section{A projector on Witt vector rings} \label{sec:2}
All rings are associative, commutative and unital. All ring homomorphisms are unital. For any commutative ring $A$ we give $A$ the discrete topology and $A^{\N}$ the product topology. Then $W (A) \equiv A^{\N}$ is a topological ring and the ghost map $\Gh : W (A) \to A^{\N}$ is continuous.

For $S$ as in the introduction, consider the ring homomorphism:
\[
\varphi_S : A^{\N} \longrightarrow A^{\N} \; , \; \varphi_S ((a_n)_{n \ge 1}) = (a_{n_S})_{n \ge 1} \; .
\]
It maps $0 \times A^{\N \setminus S}$ to zero and therefore factors over $A^S \cong A^{\N} / (0 \times A^{\N \setminus S})$. Note that $\varphi^2_S = \varphi_S$. For $k \ge 1$, Frobenius and Verschiebung maps from $A^{\N}$ to $A^{\N}$ are defined as follows:
\[
F_k ((a_n)_{n \ge 1}) = (a_{nk})_{n \ge 1} \quad \text{and} \quad V_k ((a_n)_{n \ge 1}) = k (\delta_{k \mid n} a_{n/k})_{n \ge 1} \; .
\]
Here, $\delta_{k\mid n} = 1$ if $k \mid n$ and $ = 0$ if $k \nmid n$. Now assume that $A$ is a $\Z_S$-algebra and for any prime $l \notin S$ consider the map
\[
T_l = 1 + l^{-1} V_l (1 - F_l) : A^{\N} \longrightarrow A^{\N} \; .
\]
It is a ring endomorphism because of the formula:
\[
T_l ((a_n)_{n \ge 1}) = (b_n)_{n \ge 1}
\]
where $b_n = a_n$ if $l \nmid n$ and $b_n = a_{n/l}$ if $l \mid n$. For different primes $l$ and $l'$ not in $S$, the endomorphisms $T_l$ and $T_{l'}$ commute with each other. For $m$ prime to $S$, set
\[
T_m = \prod_l T^{\ord_l m}_l \; .
\]
Then the following limit formula holds in the pointwise topology:
\begin{equation}
\label{eq:3}
\varphi_S = \lim_{\nu \to \infty} T_{m_{\nu}} \; .
\end{equation}
Here $(m_{\nu})$ is any sequence of positive integers prime to $S$ such that $m_{\nu} \mid m_{\nu +1}$ for all $\nu$ and such that any number prime to $S$ is a divisor of some $m_{\nu}$. For example, if $l_1 , l_2 , \ldots$ are the primes not in $S$ we could take $m_{\nu} = (l_1 \cdots l_{\nu})^{\nu}$.

There are also other ways to express $\varphi_S$. Firstly, we have
\[
\lim_{\nu \to \infty} T_{l^{\nu}} = \Big( \sum_{k=0}^\infty l^{-k} V_{l^k}\Big) (1-l^{-1} V_l F_l) \; .
\]
This can be either verified directly by looking at the action on
sequences or deduced from the formula for $T_l$ and the identity
$F_k\circ V_k=k$. This leads to the following formula, where the sums are over all positive integers prime to $S$:
\[
\varphi_S = \Big(\sum_{(n,S)=1} n^{-1} V_n \Big) \Big( \sum_{(n,S)=1} \mu(n) n^{-1} V_n F_n \Big) \; .
\]

Frobenius and Verschiebung operators also exist on the ring $W (A)$ of (big) Witt vectors and they correspond to Frobenius and Verschiebung on $A^{\N}$ via the ghost map $\Gh : W (A) \to A^{\N}$. If the $\Z_S$-algebra $A$ has no $\Z$-torsion, then $\Gh$ is injective and it follows from formula \eqref{eq:3} that there is a unique ring homomorphism $\varphi_S : W (A) \to W (A)$ making the diagram
\[
\xymatrix{
W (A) \ar[d]_{\varphi_S} \ar@{^{(}->}[r]^{\Gh} & A^{\N} \ar[d]^{\varphi_S} \\
W (A) \ar@{^{(}->}[r]^{\Gh} & A^{\N}
}
\]
commute. It also follows that $\varphi_S$ factors uniquely over the canonical projection $\varphi_S : W (A) \xrightarrow{\pr_S} W_S (A) \xrightarrow{\ovarphi_S} W (A)$.

We have
\begin{equation}
\label{eq:4}
\pr_S \verk \ovarphi_S = \id \; ,
\end{equation}
since this is true after applying the ghost map. 

Now let $R$ be {\it any} $\Z_S$-algebra and define $\varphi_S : W(R) \to W (R)$ by the pointwise limit \eqref{eq:3}. Convergence to a well defined ring homomorphism follows by comparison with the map $\varphi_S$ for a $\Z$-torsion free $\Z_S$-algebra $A$ surjecting onto $R$. In the same way we prove a unique factorization:
\begin{equation}
\label{eq:5}
\varphi_S : W (R) \xrightarrow{\pr_S} W_S (R) \xrightarrow{\ovarphi_S} W (R)
\end{equation}
and the formula
\begin{equation}
\label{eq:6}
\pr_S \verk \ovarphi_S = \id \quad \text{on} \; W_S (R) \; .
\end{equation}
In particular the (unital) ring homomorphism $\ovarphi_S : W_S (R) \hookrightarrow W (R)$ is injective. By construction the maps $\varphi_S$ and $\ovarphi_S$ are functorial with respect to $R$. It is clear that $\varphi^2_S = \varphi_S$. 

We need a version of the maps $\ovarphi$ for the truncation sets $S_N$: For any $\Z_S$-algebra $R$ without $\Z$-torsion, it follows by comparing with the ghost side that there is a unique ring homomorphism $\ovarphi_{S_N}$ such that the diagram
\begin{equation}
\label{eq:7}
\xymatrix{
W_S (R) \ar@{->>}[r] \ar[d]_{\ovarphi_S} & W_{S_N} (R) \ar[d]^{\ovarphi_{S_N}} \\
W (R) \ar@{->>}[r] & W_N (R)
}
\end{equation}
commutes. Here $W_N (R) = W_{\{ 1 \le \nu < N \}} (R)$. The point is that $n < N$ implies $n_S < N$. For the projection $\pr_{S_N} : W_N (R) \to W_{S_N} (R)$ we have $\pr_{S_N} \verk \ovarphi_{S_N} = \id$. Similarly as before, it follows that unique functorial ring homomorphisms $\ovarphi_{S_N}$ with the same properties exist for arbitrary $\Z_S$-algebras $R$. 
\section{Proofs of Theorems \ref{t1} and \ref{t3}} \label{sec:3}
For any ring $A$ we give the set $\Lambda (A) = 1 + t A [[t]]$ the unique ring structure, for which the bijection
\begin{equation}
\label{eq:8}
W (A) = A^{\N} \silo \Lambda (A) \; , \; (a_1 , a_2 , \ldots) \longmapsto \prod^{\infty}_{n=1} (1 - a_n t^n)
\end{equation}
is an isomorphism. Then $\Lambda (A)$ is a topological ring for the $t$-adic topology whose multiplication is uniquely determined by the formula
\[
(1 - a_1 t) \cdot (1 - a_2 t) = 1 - a_1 a_2 t \quad \text{for} \; a_1 , a_2 \in A \; .
\]
The addition in $\Lambda (A)$ is given by the multiplication of power series. We will usually view the topological isomorphism \eqref{eq:8} as an identification. There is a commutative diagram of ring homomorphisms
\begin{equation}
\label{eq:9}
\xymatrix{
W (A) \ar[r]^{\Gh} \ar@3{-}[d] & A^{\N} \ar@3{-}[d] \\
\Lambda (A) \ar[r]^{-t \partial_t \log} & tA [[t]] \; .
}
\end{equation}
Here we view $t A [[t]]$ as a commutative ring with the coefficientwise multiplication of power series, the Hadamard product. 

As before let $S \subset \N$ be divisor stable and multiplicatively closed. Let $B$ be a $\Z$-torsion free $\Z_S$-algebra with commuting Frobenius lifts $\psi^p_S$ for all primes $p \in S$. For $n \ge 1$ we set
\[
\psi^n_S = \prod_{p \in S} (\psi^p_S)^{\ord_{p^n}} \; .
\]
Let $(\lambda^n_S)$ be the special $\lambda$-ring structure on $B$ with Adams operators $\psi^n_S$, according to Lemma \ref{t2}. Consider the ring homomorphism
\[
\lambda_S : B \longrightarrow \Lambda (B)
\]
defined by the formula
\[
\lambda_S (x) = \sum^{\infty}_{i=0} (-1)^i \lambda^i_S (x) t^i \; .
\]
Setting
\[
\psi_S (x) = \sum^{\infty}_{n=1} \psi^n_S (x) t^n
\]
we have
\[
\psi_S (x) = -t \partial_t \log \lambda_S (x) \; .
\]
Using diagram \eqref{eq:9} for $A = B$ we may interpret $\lambda_S$ as the unique ring homomorphism $\talpha : B \to W(B)$ such that $\Gh \verk \talpha$ maps $x \in B$ to $(\psi^n_S (x))_{n \ge 1} \in B^{\N}$. We have $\talpha = \varphi_S \verk \talpha$ since after applying the injective ghost map, this amounts to the equality $\psi^n_S = \psi^{n_S}_S$ for $n \ge 1$. Hence we get
\[
\talpha = \varphi_S \verk \talpha = \ovarphi_S \verk \pr_S \verk \talpha = \ovarphi_S \verk \talpha_S \; .
\]
Here $\talpha_S = \pr_S \verk \talpha : B \to W_S (B)$ is the unique ring homomorphism such that $\Gh_S \verk \talpha_S$ maps $x \in B$ to $(\psi^n_S (x))_{n \in S} \in B^S$. In conclusion, we have a commutative diagram:
\begin{equation}
\label{eq:10}
\xymatrix{
B \ar[r]^{\talpha_S} \ar[d]_{\lambda_S} & W_S (B) \ar[d]^{\ovarphi_S} \\
\Lambda (B) \ar@3{-}[r] & W (B) \; .
}
\end{equation}

\begin{rem}
In the case $B = \Z_S R$ considered in Theorem \ref{t1}, the map $\talpha_S$ is the unique $\Z_S$-algebra homomorphism extending the multiplicative map $R \to W_S (\Z_S R)$ which sends $r$ to the Teichm\"uller representative of $[r]$. This follows by comparing ghost components.
\end{rem}

Let $\pi : B \to R$ be a map of $\Z_S$-algebras. For $1 \le N \le \infty$ consider the composition 
\[
\alpha_{S_N} : B \xrightarrow{\talpha_S} W_S (B) \xrightarrow{W_S (\pi)} W_S (R) \longrightarrow W_{S_N} (R) \; .
\]
For $B = \Z_S R \xrightarrow{\pi} R$, by the remark on $\talpha_S$ above, $\alpha_{S_N}$ agrees with the map \eqref{eq:1} in the introduction. Hence Theorems \ref{t1} and \ref{t3} are special cases of the following result:

\begin{theorem}
\label{t4}
With notations as above, we have
\[
\Ker \alpha_{S_N} = \{ x \in B \mid \pi \lambda^n_S (x) = 0 \quad \text{for} \; 1 \le n < N \} \; .
\]
Moreover, with $\tau_k (x) \in B \otimes \Q$ as in \eqref{eq:2}, the following formula holds in $\Lambda (B)$: 
\begin{equation}
\label{eq:11}
\lambda_S (x) = \prod_{k \in S} (1 - t^k)^{\tau_k (x)} \; .
\end{equation}
Equivalently, with notations as in Theorem \ref{t3} we have:
\[
(-1)^n \lambda^n_S (x) = \sum_{\| \nu \| = n} (-1)^{|\nu|} {\tau (x) \choose \nu} \quad \text{in $B$, for all $n \ge 1$} \; .
\]
\end{theorem}

\begin{proof}
Using diagrams \eqref{eq:7}, \eqref{eq:10} and the functoriality of $\ovarphi_S$ we get a commutative diagram
\[
\xymatrix{
B \ar[r]^{\talpha_S} \ar[d]_{\lambda_S} & W_S (B) \ar[r]^{W_S (\pi)} \ar[d]^{\ovarphi_S} & W_S (R) \ar[r] \ar[d]^{\ovarphi_S} & W_{S_N} (R) \ar[d]^{\ovarphi_{S_N}} \\
\Lambda (B) \ar@3{-}[r] & W(B) \ar[r]^{W (\pi)} & W (R) \ar[r] & W_N (R) \; .
}
\]
Identifying $W_N (R)$ with $\Lambda_N (R) = \Lambda (R) / (1 + t^N R [[t]])$, the outer square becomes
\[
\xymatrix{
B \ar[r]^{\alpha_{S_N}} \ar[d]_{\lambda_S} & W_{S_N} (R) \ar[d]^{\ovarphi_{S_N}} \\
\Lambda (B) \ar[r]^{\Lambda (\pi)} & \Lambda_N (R) \; .
}
\]
Since $\ovarphi_{S_N}$ is injective being a splitting of the projection $\pr_{S_N} : W_N (R) \to W_{S_N} (R)$, the first assertion of Theorem \ref{t4} follows. In order to prove formula \eqref{eq:11} it suffices to show the equality after applying $-t \partial_t \log$, i.e. the formula
\begin{equation}
\label{eq:12}
\psi_S (x) = \sum^{\infty}_{k=1} \tau_k (x) \frac{kt^k}{1 - t^k} \; .
\end{equation}
Generally, we have the identity of formal power series
\[
\sum^{\infty}_{k=1} a_k \frac{t^k}{1 - t^k} = \sum^{\infty}_{n=1} A_n t^n \; .
\]
where $A_n = \sum_{\nu \mid n} a_{\nu}$. In our case $a_k = k \tau_k (x)$, we obtain
\[
A_n = \sum_{\nu \mid n} \sum_{d \mid \nu} \mu (d) \psi^{\nu / d}_S (x) = \psi^n_S (x) \; ,
\]
by Moebius inversion. Thus formula \eqref{eq:12} and hence the Theorem are proved.
\end{proof}

A priori the product $\prod_{k \in S} (1-t^k)^{\tau_k (x)}$ lies in $\Lambda (B \otimes \Q)$ but its equality with $\lambda_S (x)$ shows that it lies in $\Lambda (B)$. The required integrality comes from Wilkerson's Lemma \ref{t2} and the congruences used in its proof. For $S = \{ 1 , p , p^2 , \ldots \}$ such products were considered by Dwork in his proof of Weil's rationality conjecture for zeta functions of varieties over finite fields. For $p^i \in S$ we have:
\[
\tau_1 (x) = x \quad \text{and} \quad \tau_{p^i} (x) = p^{-i} (\psi^{p^i}_S (x) - \psi^{p^{i-1}}_S (x)) \; \text{for} \; i \ge 1 \; .
\]
Thus equation \eqref{eq:11} asserts:
\begin{equation}
\label{eq:13}
\lambda_S (x) = (1 - t)^x \prod^{\infty}_{i=1} (1 - t^{p^i})^{\tau_{p^i} (x)} \; . 
\end{equation}

In \cite{D} p. 2, using slightly different notation, Dwork considers the following product in the formal power series ring $\Q [[t,X]]$:
\begin{equation}
\label{eq:14}
F (X,t) = (1+t)^X \prod^{\infty}_{i=1} (1 + t^{p^i})^{p^{-i} (X^{p^i} - X^{p-1})} \; .
\end{equation}
Using his well-known criterion \cite{D} Lemma 1, he shows that the coefficients of $F (X,t)$ are $p$-integral. Our sign conventions concerning $\Lambda$-and Witt rings are not quite compatible with Dwork's. However, for odd $p$ we can relate $F (X,t)$ to $\lambda_S$ and $\varphi_S$ as follows. For $S$ as above, we have $\Z_S = \Z [l^{-1} \mid l \neq p]$. Equip the $\Z$-torsion free $\Z_S$-algebra $B = \Z_S [X]$ with the Frobenius lift $\psi^p_S$ defined by $\psi^p_S (X) = X^p$. Then, as in the beginning of this section the corresponding $\lambda$-ring structure on $\Z_S [X]$ is encoded in a ring homomorphism
\[
\lambda_S : \Z_S [X] \longrightarrow \Lambda (\Z_S [X]) = 1 + t \Z_S [X] [[t]] \subset \Z_S [[t,X]] \; .
\]
Comparing \eqref{eq:13} and \eqref{eq:14}, we see that for $p \neq 2$ we have:
\begin{equation}
\label{eq:15}
F (X, -t) = \lambda_S (X) \; .
\end{equation}
In particular the $p$-integrality of \eqref{eq:14} follows. In terms of the map 
\[
\ovarphi_S : W_S (\Z_S [X]) \to W (\Z_S [X]) \equiv \Lambda (\Z_S [X])
\]
we have
\[
F (X , -t) = \ovarphi_S (\langle X \rangle) \; .
\]
Here $\langle X \rangle$ is the Teichm\"uller representative of $X$. This follows from diagram \eqref{eq:10} and formula \eqref{eq:15} noting that $\talpha_S (X) = \langle X \rangle$. The latter equality holds because $(\Gh_S \verk \talpha_S) (X) = (\psi^{p^i}_S (X)) = (X^{p^i})$ by the characterization of $\talpha_S$ and since $\Gh_S (\langle X \rangle) = (X^{p^i})$. 
%\bibliographystyle{alpha}
%\bibliography{bib}

\begin{thebibliography}{Dwo67}

\bibitem[BW05]{BW}
James Borger and Ben Wieland.
\newblock Plethystic algebra.
\newblock {\em Adv. Math.}, 194(2):246--283, 2005.

\bibitem[CD14]{CD1}
Joachim Cuntz and Christopher Deninger.
\newblock An alternative to {W}itt vectors.
\newblock {\em M\"unster J. Math.}, 7(1):105--114, 2014.

\bibitem[CD15]{CD2}
Joachim Cuntz and Christopher Deninger.
\newblock Witt vector rings and the relative de {R}ham {W}itt complex.
\newblock {\em J. Algebra}, 440:545--593, 2015.
\newblock With an appendix by Umberto Zannier.

\bibitem[Dwo67]{D}
B.~Dwork.
\newblock On the rationality of zeta functions and {$L$}-series.
\newblock In {\em Proc. {C}onf. {L}ocal {F}ields ({D}riebergen, 1966)}, pages
  40--55. Springer, Berlin, 1967.

\bibitem[Hes15]{H}
Lars Hesselholt.
\newblock The big de {R}ham-{W}itt complex.
\newblock {\em Acta Math.}, 214(1):135--207, 2015.

\bibitem[Rab14]{R}
Joseph Rabinoff.
\newblock The {T}heory of {W}itt {V}ectors, 2014.
\newblock arXiv:1409.7445.

\bibitem[Wil82]{W}
Clarence Wilkerson.
\newblock Lambda-rings, binomial domains, and vector bundles over {${\bf
  C}P(\infty )$}.
\newblock {\em Comm. Algebra}, 10(3):311--328, 1982.

\end{thebibliography}

\end{document}